\documentclass[12pt]{amsart}

\usepackage{graphicx}
\usepackage{amsthm}
\swapnumbers

\newtheorem{Thm}{Theorem}

\newtheorem{Coro}[Thm]{Corollary}
\newtheorem{Lem}[Thm]{Lemma}

\begin{document}

\title{Generalized handlebody sets and non-Haken 3-manifolds}
\author{Jesse Johnson and Terk Patel}

\thanks{Research supported by NSF MSPRF grant 0602368}

\begin{abstract}
In the curve complex for a surface, a handlebody set is the set of loops that bound properly embedded disks in a given handlebody bounded by the surface.  A boundary set is the set of non-separating loops in the curve complex that bound two-sided, properly embedded surfaces.  For a Heegaard splitting, the distance between the boundary sets of the handlebodies is zero if and only if the ambient manifold contains a non-separating, two sided incompressible surface.  We show that the boundary set is 2-dense in the curve complex, i.e. every vertex is within two edges of a point in the boundary set.
\end{abstract}

\maketitle

\section{Introduction}

The curve complex $C(\Sigma)$ for a compact, connected, closed, orientable surface $\Sigma$ is the simplicial complex whose vertices are loops (isotopy classes of essential, simple closed curves) in $\Sigma$ and whose simplices correspond to sets of pairwise disjoint loops in $\Sigma$.  Given a handlebody $H$ and a homeomorphim $\phi : \Sigma \rightarrow \partial H$, we can define the following subsets of $C(\Sigma)$:

The \textit{handlebody set} $\mathbf{H}$ is the set of loops that bound properly embedded (essential) disks in $H$.  The \textit{genus $g$ boundary set} $\mathbf{H}^g$ is the set of non-separating loops such that each bounds a properly embedded, two-sided, incompressible, genus-$g$ surface in $H$.  Note that $\mathbf{H}^0$ is a proper subset of $\mathbf{H}$, specifically the set of all the non-separating loops in $\mathbf{H}$.  Define the \textit{boundary set} to be the union $\mathbf{H}^\infty = \bigcup_{g\geq0} \mathbf{H}^g$.  We prove the following:

\begin{Thm}
\label{2denselem}
If $\Sigma$ has genus 3 or greater, then $\mathbf{H}^\infty$ is 2-dense in $C(\Sigma)$.
\end{Thm}

The proof presented here does not work for genus two surfaces.  However, Schleimer~\cite{schl:notes} has shown that the orbit of a vertex of $C(\Sigma)$ under the action of the Torelli group is 5-dense.  This imples that for a genus two handlebody, $\mathbf{H}^{\infty}$ is at most $n$-dense for some $n \leq 5$.  

In contrast to $\mathbf{H}^\infty$, a fixed genus boundary set $\mathbf{H}^g$ has a geometric structure much closer to $\mathbf H$, which is not $k$-dense for any $k$.  This is demonstrated by the following two Lemmas, the first of which can be proved with a simple construction and the second of which follows from a Theorem of Scharlemann~\cite{schar:dist}.  The proofs of both are left to the reader.

\begin{Lem}
\label{neartodiskslem}
If $\Sigma$ has genus three or greater and $v \in \mathbf{H}$ then $d(v, \mathbf{H}^g) = 1$ for every $g > 0$.  If $\Sigma$ has genus two then $d(v, \mathbf{H}^g) > 1$ for every $g$.
\end{Lem}

\begin{Lem}
For $g \geq 1$, the set $\mathbf{H}^g$ is disjoint from $\mathbf{H}$ and contained in a $2g$ neighborhood of $\mathbf{H}$.
\end{Lem}

For this paper, every 3-manifold will be compact, connected, closed and orientable.  A \textit{Heegaard splitting} for such a 3-manifold $M$ is a triple $(\Sigma, H_1, H_2)$ where $\Sigma \subset M$ is a compact, connected, closed, orientable surface and $H_1, H_2 \subset M$ are handlebodies such that $\partial H_1 = \Sigma = \partial H_2$ and $M = H_1 \cup H_2$.  The inclusion maps from $\partial H_1$ and $\partial H_2$ onto $\Sigma$ determine handlebody sets $\mathbf{H}_1$ and $\mathbf{H}_2$, respectively.  The \textit{distance} of the Heegaard splitting, as defined by Hempel~\cite{hemp:comp} is the distance $d(\Sigma) = d(\mathbf{H}_1, \mathbf{H}_2)$ between the two handlebody sets.  

The inclusion maps also determine boundary sets $\mathbf{H}^g_1$, $\mathbf{H}^h_2$, $\mathbf{H}^\infty_1$, $\mathbf{H}^\infty_2$, allowing us to generalize this distance to the \textit{$(g,h)$-distance} $d^{g,h}(\Sigma) = d(\mathbf{H}^g_1, \mathbf{H}^h_2)$ and the \textit{boundary distance} $d^\infty(\Sigma) = d(\mathbf{H}^\infty_1, \mathbf{H}^\infty_2)$.

The set $\mathbf{H}^{\infty}$ is precisely the set of vertices representing simple closed curves whose homology class is non-trivial in $\Sigma$, but trivial in $H$.  This definition can be thought of as replacing every instance of homotopy in the definition of a handlebody set with homology.  It thus encodes information about the homology of the handlebody.  For two handlebodies, it encodes homology information about the ambient manifold.  In particular, the boundary distance determines precisely when a manifold has infinite homology (and therefore a non-separating, incompressible surface.)

\begin{Lem}
\label{zerolem}
The following are equivalent:

(1) the first homology group of $M$ is infinite,

(2) $M$ contains a non-separating, two sided, closed incompressible surface,

(3) $d^\infty(\Sigma) = 0$ and

(4) $d^{0,\infty}(\Sigma) = 0$.
\end{Lem}

The proof is given in Section~\ref{constructionsect}.  The equivalence of (1) and (2) is well known, but we give a very simple, geometric proof via the boundary set.  Theorem~\ref{2denselem} is proved in Section~\ref{2densesect}.

For any Heegaard splitting $(\Sigma, H_1, H_2)$ of a non-Haken 3-manifold, Lemma~\ref{zerolem} implies that the boundary set in $C(\Sigma)$ determined by $H_2$ must be completely disjoint from the boundary set for $H_1$.  Hempel showed that there are handlebody sets that are arbitrarily far apart in the curve complex.  The same is not true for boundary sets.  In particular, Theorem~\ref{2denselem} implies that for any Heegaard splitting $(\Sigma, H_1, H_2)$ of genus 3 or greater, $d^\infty(\Sigma)$ is equal to either 0, 1 or 2.  For non-Haken manifolds, we have the following:

\begin{Coro}
For any Heegaard splitting $(\Sigma, H_1, H_2)$ of a non-Haken 3-manifold $M$, $d^\infty(\Sigma)$ is equal to 1 or 2.
\end{Coro}

\section{Non-separating surfaces}
\label{constructionsect}

The following Lemma will not be used until Section~\ref{2densesect}, but the method of proof gives a good introduction to the proof of Lemma~\ref{infhomlem}.  Recall that an element $\alpha$ of a $\mathbf{Z}$ module $G$ is called \textit{primitive} if there is no $\beta \in G$ such that $\alpha = k \beta$ for some $k \neq \pm 1$.

\begin{Lem}
\label{dsjtsurfacelem}
Let $\ell_1,\dots,\ell_k$ be pairwise disjoint, essential loops in the boundary of a genus-$g$ handlebody $H$ with $g > k$.  Then there is a properly embedded, non-separating surface $F \subset H$ such that $\partial F$ is disjoint from each $\ell_j$ and the homology class defined by $\partial F$ in $H(\Sigma)$ is primitive.
\end{Lem}

\begin{proof}
Let $D_1,\dots,D_g$ be a system of disks for $H$, i.e. a collection of properly embedded, essential disks whose complement in $H$ is a single ball.  Orient the boundaries of the disks and the loops $\ell_1,\dots,\ell_k$, then form the matrix $A = (a_{ij})$ such that $a_{ij}$ is the algebraic intersection number of $D_i$ and $\ell_j$.

If we replace one of the disks in the system by a disk slide, the matrix for the new system of disks can be constructed from $A$ by adding or subtracting one row from the other.  Thus we can perform elementary row operations on $A$ by choosing new systems of disks for $H$.  In particular, we can make $A$ upper triangular.  

Because $A$ has more rows than columns, if $A$ is upper triangular then the bottom row consists of all zeros.  In other words, the disk $D_g$ has algebraic intersection 0 with each loop $\ell_j$.

If $D_g$ intersects the loop $\ell_1$, there must be a pair of adjacent intersections with opposite orientations.  By attaching a band from $\partial D_g$ to itself, along the arc of the loop $\ell_1$, we can form a new surface $F_1$ whose boundary (consisting of two loops) has algebraic intersection zero with each loop $\ell_j$, but whose geometric intersection number is strictly lower than that of $D_g$.  The surface $F_1$ is two sided because the intersections have opposite orientations and non-separating because $D_g$ is non-separating.

If $\partial F_1$ intersects $\ell_1$, we can form a new surface $F_2$ by attaching a band, and so on.  Continuing in this manner for each $\ell_j$, we form a surface $F$ which is properly embedded, two-sided, non-separating and such that $\partial F$ is disjoint from each $\ell_j$.

Attaching a band to the boundary of $F_i$ does not change the homology class of the boundary, so the homology class of $\partial F$ is equal to the class of $\partial D_g$.  Because $\partial D_g$ is represented by a connected loop, its homology class is primitive, as is the homology class of $\partial F$.
\end{proof}

We will now use the idea of attaching bands to eliminate intersections to prove the implication $(1) \Rightarrow (4)$ of Lemma~\ref{zerolem}.

\begin{Lem}
\label{infhomlem}
If the first homology of $M$ is infinite then $d^{0,\infty}(\Sigma) = 0$.
\end{Lem}

\begin{proof}
Let $D_1,\dots,D_g$ be a system of disks for $H_1$ and $D'_1,\dots,D'_g$ be a system of disk for $H_2$.  Orient the boundaries of both systems of disks.  Let $A$ be the matrix of algebraic intersection numbers of the boundaries.  Because the first homology is infinite, the determinant of $A$ must equal zero.  

As in the last proof, we can perform row operations on $A$ by taking disk slides of the disks $D_1,\dots,D_g$.  Because the determinant of $A$ is zero, some sequence of disk slides will leave $A$ with all zeros in the bottom row.  Thus after a sequence of disk slides, we can assume $D_g$ has algebraic intersection 0 with each $D'_j$.

By attaching bands to the boundary of $D_g$ as in the proof of Lemma~\ref{dsjtsurfacelem}, we can form a properly embedded, two sided, non-separating surface $F$ whose boundary is disjoint from $D'_1,\dots,D'_g$.  Thus each boundary component of $F$ bounds a disk in $H_2$.  The union of $F$ and these disks is a properly embedded, two sided, non-separating closed surface in $M$.

Recall that $F$ was constructed from $D_g$ by attaching bands to its boundary.  The last band defines a boundary compression for $F$ corresponding to an isotopy pushing this last band into $H_2$.  After this isotopy, the second to last band defines a second isotopy, and so on.  The final result is a surface isotopic to $F$ which intersects $H_1$ in a disk isotopic to $D_g$.  

The intersection of this surface with $H_2$ is orientable, two-sided and non-separating because $F$ has these properties.  Thus $\partial D_g$ is in both  $\mathbf{H}^0_1$ and $\mathbf{H}^\infty_2$ so $d^{0,\infty}(\Sigma) = d(\mathbf{H}_1^0, \mathbf{H}_2^\infty) = 0$.
\end{proof}

\begin{proof}[Proof of Lemma~\ref{zerolem}]
Lemma~\ref{infhomlem} implies that for any Heegaard splitting $(\Sigma, H_1, H_2)$, $d(\mathbf{H}_1^0, \mathbf{H}_2^\infty) = 0$ so $(1) \Rightarrow (4)$.  Because $\mathbf{H}_1^0$ is contained in $\mathbf{H}_1^\infty$, $(4) \Rightarrow (3)$ is immediate.

Let $(\Sigma, H_1, H_2)$ be a Heegaard splitting for $M$.  If $d^\infty(\Sigma) = 0$ then there is a simple closed curve $\ell \subset \Sigma$ such that $\ell$ bounds two-sided, non-separating properly embedded surfaces $F \subset H_1$ and $F' \subset H_2$.  The union $F \cup F'$ is a two-sided, non-separating closed surface embedded in $M$.  Compressing $F \cup F'$ to either side produces at least one new two-sided, non-separating surface.  By compressing repeatedly, we eventually find a closed, non-separating, two-sided incompressible surface in $M$.  Thus $(3) \Rightarrow (2)$.

The final step, $(2) \Rightarrow (1)$, is a classical result.  If $M$ contains a two-sided, non-separating, closed surface $S \subset M$, let $p$ be a point in $S$.  There is a path $\alpha : [0,1] \rightarrow M$ from $p$ to itself that does not cross $S$.  The homology class of $\alpha$ has infinite order so the first homology of $M$ is infinite. 
\end{proof}

\section{Density}
\label{2densesect}

\begin{proof}[Proof of Theorem~\ref{2denselem}]
We will prove the following:  Let $\ell$ be a loop in $\partial H$ and assume the genus of $H$ is at least 3.  Then there is a non-separating loop $\ell'$ disjoint from $\ell$ and a properly embedded, two-sided, non-separating surface $F$ such that $\partial F$ is a single, non-separating loop disjoint from $\ell'$.

By Lemma~\ref{dsjtsurfacelem}, there is a properly embedded surface $F'' \subset H$ such that $\partial F''$ is disjoint from $\ell$ and defines a primitive element of the homology.  Of all the properly embedded surfaces with boundary homologous to $\partial F''$, let $F'$ be one with minimal number of boundary components.  Each component of $\partial F'$ has an orientation induced by $F'$ and thus defines an element of the first homology of $\Sigma$.

For each component $C$ of $\Sigma \setminus (\ell \cup \partial F')$, an orientation for a loop in $\partial C$ induces an orientation of $C$.  Assume two components of $\partial C$ come from loops of $\partial F'$ and induce the same orientation of $C$.  Because the induced orientations agree, adding a band between them produces a new orientable surface with fewer boundary components, but homologous boundary.  Thus the minimality assumption implies that each component $C$ of $\Sigma \setminus (\ell \cup \partial F')$ has at most one boundary loop coming from $\partial F'$ inducing each possible orientation.  Thus it has at most two boundary loops coming from $\partial F'$ and these induce opposite orientations on $C$.

Assume for contradiction each component of $\Sigma \setminus (\ell \cup \partial F')$ is planar.  Each component that is disjoint from $\ell$ has exactly two boundary components.  A planar surface with two boundary loops is an annulus so each component disjoint from $\ell$ must be an annulus.  There are either two components of $\Sigma \setminus (\ell \cup \partial F')$ with one boundary loop each on $\ell$, or one component with two boundary loops on $\ell$.  In the first case, the two components are pairs of pants, while in the second, the component is a four punctured sphere.  In either case, the union of such components and a collection of annuli is a genus-two surface.  This contradicts the assumption that $\Sigma$ has genus at least three, so we conclude that some component must be non-planar.

Let $C$ be a non-planar component.  There is a simple closed curve $\ell' \subset C$ such that $\ell'$ separates a once-punctured torus from $C$.  In $\Sigma$, the loop $\ell'$ separates a once-punctured torus that contains no components of $\partial F'$.  Let $F$ be a surface whose boundary is homologous to $F'$, disjoint from the once-punctured torus bounded by $\ell'$ and such that the number of boundary components of $F'$ is minimal over all such surfaces.

Once again, each component of $\Sigma \setminus (\ell' \cup \partial F)$ has at most two boundary components on loops in $\partial F$, with opposite induced orientations.  Because $\ell'$ bounds a surface disjoint from $\partial F$, each component of $\Sigma \setminus \partial F$ must also have at most two boundary loops on $\partial F$.  

If a component $C$ of $\Sigma \setminus \partial F$ has a single boundary component, this loop is homology trivial in $\Sigma$.  Attaching a boundary parallel surface to $F$ removes this loop so minimality of $\partial F$ implies that that each component has two boundary loops.

If $C$ has two boundary loops (with opposite induced orientations) then these loops determine the same element of the homology of $\Sigma$.  Because $\Sigma$ is connected, this implies that any two loops of $\partial F$ (with their induced orientations) determine the same element of the homology.  Thus the element of the homology determined by $\partial F$ is of the form $k \beta$ where $k$ is the number of boundary components of $F$.

By Lemma~\ref{dsjtsurfacelem}, $\partial F$ determines a primitive element of the homology of $\Sigma$, so $k$ must be 1.  In other words, the boundary of $F$ is connected and $\partial F$ determines an element of $\mathbf{H}^\infty$.  By construction, $\partial F$ is disjoint from a loop $\ell'$ that is disjoint from $\ell$.  Thus the vertex $v \in C(\Sigma)$ determined by $\ell$ is distance at most 2 from $\mathbf{H}^\infty$.
\end{proof}

\bibliographystyle{amsplain}
\bibliography{bndry}

\end{document}